\documentclass[12pt]{article}

\usepackage[usenames]{color}
\usepackage{amssymb}
\usepackage{amsmath}
\usepackage{amsthm}
\usepackage{amsfonts}
\usepackage{amscd}
\usepackage{graphicx}

\usepackage{bm}

\usepackage[colorlinks=true,
linkcolor=webgreen,
filecolor=webbrown,
citecolor=webgreen]{hyperref}

\definecolor{webgreen}{rgb}{0,.5,0}
\definecolor{webbrown}{rgb}{.6,0,0}

\usepackage{float}
\usepackage{latexsym}
\usepackage{epsf}

\newtheorem{Theorem}{Theorem}[section]

\newtheorem{Proposition}[Theorem]{Proposition}

\theoremstyle{definition}
\newtheorem{Remark}[Theorem]{Remark}

\newtheorem{Example}[Theorem]{Example}

\def\QQ{{\mathbb Q}}
\def\ZZ{{\mathbb Z}}
\def\RR{{\mathbb R}}

\begin{document}

\begin{center}
\vskip 1cm{\Large\bf The $h^{*}$-polynomial of the cut polytope of $K_{2,m}$ in the lattice spanned by its vertices}
\vskip 1cm
\large
Ryuichi Sakamoto\\
Department of Mathematical Sciences\\
Graduate School of Science and Technology\\
Kwansei Gakuin University\\
Sanda, Hyogo 669-1337, Japan\\
\href{mailto:dpm86391@kwansei.ac.jp}{\tt dpm86391@kwansei.ac.jp}
\end{center}

\begin{abstract}
The cut polytope of a graph is an important object in several fields, such as functional analysis, combinatorial optimization, and probability.
For example, Sturmfels and Sullivant showed that the toric ideals of cut polytopes are useful in algebraic statistics. 
In the theory of lattice polytopes, the $h^{*}$-polynomial is an important invariant.
However, except for trees, there are no classes of graphs for which the $h^{*}$-polynomial of their cut polytope is explicitly specified. 
In the present paper, we determine the $h^{*}$-polynomial of the cut polytope of complete bipartite graph $K_{2,m}$ using the theory of Gr\"{o}bner bases of toric ideals.
\end{abstract}

\section{Introduction}
Given integer vectors $\bm{v}_1, \ldots, \bm{v}_l \in \ZZ^{d}$,  
let $${\rm conv}( \bm{v}_1, \ldots, \bm{v}_l )=\left\{\sum_{i=1}^{l}r_i\bm{v}_{i} \ \left|  \ 0\leq{r_i}\in \RR, \sum_{i=1}^{l}r_i=1 \right.\right\}.$$
A set $\mathcal{P}\subset \RR^d$ is called a {\em lattice polytope}
if there exist $\bm{v}_1, \ldots, \bm{v}_l \in \ZZ^{d}$ such that
$\mathcal{P} = {\rm conv}( \bm{v}_1, \ldots, \bm{v}_l )$.
Let $\mathcal{P}\subset \RR^d$ be a lattice polytope of dimension $d$, where 
$\mathcal{P}\cap \ZZ^{d}=\{\bm{a}_1, \ldots, \bm{a}_n\}$.
Let $\mathcal{A_{\mathcal{P}}}$ be the integer matrix 
$$\mathcal{A_{\mathcal{P}}}=
\begin{pmatrix}
\bm{a}_1 & \cdots & \bm{a}_{n}\\
1 & \cdots & 1
\end{pmatrix}.
$$
The {\em normalized Ehrhart polynomial} $i(\mathcal{P},m)$ is given by the following equation:
$$i(\mathcal{P},m)=|m\mathcal{P}' \cap \ZZ \mathcal{A_{\mathcal{P}}}|,$$
where $m\in \mathbb{N}$, 
$\mathcal{P}'={\rm conv}(\binom{\bm{a}_1}{1}, \ldots, \binom{\bm{a}_n}{1})$, 
$m \mathcal{P}' = \{m\bm{a} \ | \ \bm{a}\in \mathcal{P}'\}$, and 
$$\ZZ \mathcal{A_{\mathcal{P}}} = \ZZ \begin{pmatrix}\bm{a}_1\\1\end{pmatrix}+\cdots +\ZZ \begin{pmatrix}\bm{a}_n\\1\end{pmatrix}. $$
In general, $i(\mathcal{P},m)$ satisfies the following fundamental properties \cite{Ehr}:
\begin{itemize}
\item
$i(\mathcal{P},m)$ is a polynomial of degree $d$ in $m$;
\item
$i(\mathcal{P},0)=1$.
\end{itemize}
The {\em $h^{*}$-polynomial} $h^{*}(\mathcal{P},x)$ of $\mathcal{P}$ in the lattice $\ZZ \mathcal{A_{\mathcal{P}}}$ is defined by 
$$1+\sum_{m=1}^{\infty}i(\mathcal{P}, m)x^m=\frac{h^{*}(\mathcal{P},x)}{(1-x)^{d+1}}.$$ 
In general, $h^{*}(\mathcal{P},x)$ satisfies the following properties:
\begin{itemize}
\item
$h^{*}(\mathcal{P},x)=\sum_{i=0}^{d}h_{i}^{*}x^i$,
where each $h_{i}^{*}$ is a nonnegative integer \cite{Stanley};
\item
$i(\mathcal{P},m)=\sum_{i=0}^{d}h_{i}^{*}\binom{m+d-i}{d}$;
\item
If $h_{d}^{*}>0$, then we have $h_{i}^{*}\ge h_{1}^{*}$ $(1\leq{i}\leq{d-1})$ \cite[Theorem~1.1]{Hibi}.
\end{itemize}
The third property is {\em Hibi's lower bound theorem}.
Since the $h^{*}$-polynomial is defined in the lattice $\ZZ \mathcal{A_{\mathcal{P}}}$, ${\mathcal P}$ is a {\em spanning lattice polytope} \cite{spanning}, and a generalization of Hibi's lower bound theorem is known:
\begin{itemize}
\item
$h_{i}^{*}\ge h_{1}^{*}$ $(1 \leq{i} \leq \deg ( h^{*}(\mathcal{P},x)) -1)$ \cite[Corollary~1.6]{HKN}.
\end{itemize}
A polynomial $f(x)$ of degree $s$ is said to be {\it palindromic} if $f(x)=x^sf(x^{-1})$.
Let $K[\mathcal{A_{\mathcal{P}}}]$ be the toric ring of $\mathcal{P}$.
(The toric ring will be defined in Section $1$.)
If $K[\mathcal{A_{\mathcal{P}}}]$ is normal 
(i.e., $\ZZ_{\ge 0}\mathcal{A_{\mathcal{P}}}=\ZZ \mathcal{A_{\mathcal{P}}}\cap \QQ_{\ge 0}\mathcal{A_{\mathcal{P}}}$, see \cite[Proposition~13.5]{Sturm}) and Gorenstein, then it is known by \cite[Lemma $4.22 \ (b)$]{HHO} and \cite[P.235]{Bruns} that the $h^{*}$-polynomial of ${\mathcal P}$ is
palindromic.
In the theory of lattice polytopes, $h^{*}$-polynomials
are important objects to study.
For example, the $h^{*}$-polynomials of stable set polytopes of graphs, order polytopes, and chain polytopes of posets were studied in \cite{Atha, Stanley2}.
Let $G$ be a finite connected simple graph with vertex set $V(G)=\{1,2,\ldots ,m\}$ and edge set $E(G)=\{e_{1},e_{2},\ldots ,e_{r}\}$.
For two subsets $A$ and $B$ of $V(G)$ such that $A\cap B=\emptyset, A\cup B=V(G)$,
we define a vector $\delta_{A|B}=(d_{1},d_{2},\ldots ,d_{r})\in \{0,1\}^r$ by
\begin{equation*}
d_{i}=\begin{cases}
								1 & \text{$|A\cap e_{i}|=1$,}\\
								0 & \text{otherwise.}
							\end{cases}
\end{equation*}
Note that $\delta_{A|B}=\delta_{B|A}$.
The {\em cut polytope} of $G$ is the $0/1$ polytope
$${\rm Cut}(G)={\rm conv}(\delta_{A|B} \hspace{2truept} |A,B\subset V(G), A\cap B=\emptyset, A\cup B=V(G) ).$$
\begin{Example}
Let $G$ be a cycle of length $4$, 
where $V(G)=\{1,2,3,4\}$ and $E(G)=\{e_1=\{1,2\}, e_2=\{2,3\}, e_3=\{3,4\}, e_4=\{1,4\}\}$. 
For subset $A=\{1,2\}\subset V(G)$, we compute $\delta_{A|B}=(d_1,d_2,d_3,d_4)$, where $B=\{3,4\}$. Since $|A\cap e_1|=2$, we have $d_1=0$. Similarly, we obtain $d_2=1, d_3=0$, and $d_4=1$. Hence, $\delta_{A|B}=(0,1,0,1)$. By computing $\delta_{A|B}$ for all subsets $A\subset V(G)$, we obtain the cut polytope ${\rm Cut}(G)={\rm conv}((0,0,0,0), (1,1,0,0), (1,0,1,0), (1,0,0,1), $\\$(0,1,1,0), (0,1,0,1),(0,0,1,1), (1,1,1,1)).$   
\end{Example}

We define the graph theoretical terminology used in the present paper. A {\it bridge} of a graph $G$ is an edge of $G$
whose deletion increases the number of connected components, and a graph $G$ is said to be {\it bridgeless} if $G$ has no bridges. An {\it induced cycle} of $G$ is a cycle 
of $G$ that is an induced subgraph of $G$. 
A graph $G$ is said to be {\it chordal} if $G$ has no induced cycles of length $\ge 4$. 
A graph $H$ is called a {\it minor} of a graph $G$ if $H$ is obtained from $G$ by a sequence of contractions and deletions of edges. 
On the other hand, if we cannot obtain $H$ as a minor of $G$, then
$G$ is said to be {\it $H$-minor free}.
A complete graph with $n$ vertices is denoted by $K_n$, a complete bipartite graph with $m+n$ vertices is denoted by $K_{m,n}$, and a cycle of length $n$ is denoted by $C_n$.
There is only one class of cut polytopes for which the $h^{*}$-polynomial is explicitly known.
Nagel and Petrovi\'{c}
 \cite{Petro} showed that,
if $G$ is a tree with $n\geq 1$ edges,
then the $h^{*}$-polynomial of the cut polytope in the lattice $\ZZ {\mathcal A}_{{\rm Cut}(G)}$ 
 is the Eulerian polynomial $$A_{n}(x):={\displaystyle \sum_{w\in\mathfrak{S}_n}x^{{\rm des}(w)}}$$ of degree $n-1$.
Here,  $\mathfrak{S}_n$ is a symmetric group and ${\rm des}(w)=|\{i \ | \ w_i>w_{i+1}\}|$ for $w=w_1w_2\cdots w_n\in \mathfrak{S}_n$. It is known that $A_{n}(x)$ is palindromic and unimodal.
Ohsugi \cite{OHSUGI} showed that
the toric ring of the cut polytope ${\rm Cut}(G)$ of a graph $G$ is normal and Gorenstein
if and only if $G$ is $K_5$-minor free and satisfies one of the following:
\begin{enumerate}
\item
$G$ is a bipartite graph with no induced cycle of length $\ge 6$.
\item
$G$ is a bridgeless chordal graph.
\end{enumerate}  
Thus, if $G$ satisfies one of the above conditions, then the $h^{*}$-polynomial of the cut polytope of $G$
is palindromic, since the toric ring is normal and Gorenstein.

In the present paper,
we determine the $h^{*}$-polynomial of the cut polytope of a complete bipartite graph $K_{2,n-2}$ and show that the $h^{*}$-polynomial is $(x+1) (A_{n-2}(x))^2$
using the theory of Gr\"{o}bner bases of toric ideals. See \cite{HHO, Sturm} for the details of Gr\"{o}bner bases and toric ideals. 
\begin{Remark}
We discuss the normalized Ehrhart polynomial $$i(\mathcal{P},m)=|m\mathcal{P}' \cap \ZZ\mathcal{A_{\mathcal{P}}}|$$
instead of the ordinary {\it Ehrhart polynomial} $$|m\mathcal{P}\cap \ZZ^d|$$
because the lattice spanned by $\delta_{A|B}$'s is important in the study of the cut polytopes. 
In fact, the characterization of the graph $G$ satisfying $\ZZ_{\ge 0}B_G=\ZZ B_G\cap \QQ_{\ge 0}B_G$, where $B_G={\rm Cut}(G)\cap \ZZ^d$, is an important open problem. See \cite{LTT} and the references therein.
In addition, it has been conjectured that $\ZZ_{\ge 0}\mathcal{A}_{{\rm Cut}(G)}=\ZZ \mathcal{A}_{{\rm Cut}(G)}\cap \QQ_{\ge 0}\mathcal{A}_{{\rm Cut}(G)}$ if and only if $G$ is $K_5$-minor free \cite{SturmfelsSullivant}. See \cite{Ohsugi1, SturmfelsSullivant}. 
\end{Remark}

\section{Standard monomials of cut ideals of $K_{2,n-2}$}
Let $K[\bm{x}]=K[x_1, \ldots ,x_n]$ be a polynomial ring in $n$ variables over a field $K$.
Let $\mathcal{M}_n$ be the set of all monomials of $K[\bm{x}]$. A total order $<$ on $\mathcal{M}_n$ is called a {\it monomial order} if $<$ satisfies the following properties:
\begin{itemize}
\item
For all $1\neq u\in \mathcal{M}_n$, $1<u.$
\item
If $u<v$ ($u,v\in \mathcal{M}_n$), then $w\cdot u<w\cdot v$ for all $w\in\mathcal{M}_n$.
\end{itemize}
We fix a monomial order $<$. 
For a nonzero polynomial $f$ which belongs to $K[\bm{x}]$, the {\em initial monomial} in$_{<}(f)$ of $f$ is the greatest monomial in $f$ with respect to $<$. The {\it initial ideal} of an ideal $I \subset K[\bm{x}]$ with respect to $<$ is defined by in$_{<}(I)=\langle$ in$_{<}(f) \ | \ 0\neq f\in I \rangle$. 
A finite subset $G=\{g_1, \ldots, g_s\} \subset I$ is called a {\it Gr\"{o}bner basis} of $I$ with respect to $<$ when in$_{<}(I)=\langle$in$_{<}(g_1), \ldots ,$in$_{<}(g_s) \rangle$. 
See \cite[Chapter 1]{HHO} for the basic theory of Gr\"{o}bner bases.

A $d \times n$ integer matrix $A=(\bm{a}_{1},\bm{a}_{2},\ldots ,\bm{a}_{n})$ is called a {\em configuration} if there exists a vector $\bm{c}\in \mathbb{R}^d$ such that for all $1\leq{i}\leq{n}$, the inner product $\bm{a}_{i} \cdot \bm{c}$ is equal to $1$.
Let $K[t_1^{\pm 1},t_2^{\pm 1},\ldots,t_d^{\pm 1}]$ be a Laurent polynomial ring over $K$. 
For an integer vector $\bm{b}=(b_{1},b_{2},\ldots ,b_{d})\in \mathbb{Z}^d$,
we define the Laurent monomial $\bm{t^{b}}=t_{1}^{b_{1}}t_{2}^{b_{2}}\cdots t_{d}^{b_{d}} \in K[t_1^{\pm 1},t_2^{\pm 1},\ldots,t_d^{\pm 1}]$ and the {\it toric ring} 
$K[A]=K[\bm{t}^{\bm{a}_{1}},\bm{t}^{\bm{a}_{2}},\ldots ,\bm{t}^{\bm{a}_{n}}]$. Let $\pi$ be a homomorphism 
$\pi:K[\bm{x}]\rightarrow K[A]$, where $\pi (x_{i})=\bm{t}^{\bm{a}_{i}}$. The kernel of $\pi$ is called the {\em toric ideal} of $A$ and denoted by $I_{A}$. We often regard $A=(\bm{a}_1,\ldots, \bm{a}_n)$ as a set $A=\{\bm{a}_1,\ldots,\bm{a}_n\}$. 
Suppose that a set $\Delta$ consists of simplices and that each vertex of $\sigma \in \Delta$ belongs to $A$. Then, $\Delta$ is called a {\it covering} of ${\rm conv}(A)$ if $${\rm conv}(A)={\displaystyle \bigcup_{F\in \Delta}F}.$$
 We say that a covering $\Delta$ is a {\it triangulation} of ${\rm conv}(A)$ if 
$\Delta$ is a simplicial complex. A triangulation $\Delta$ of a polytope ${\rm conv}(A)$ is {\it unimodular} if the normalized volume of 
any maximal simplex is equal to $1$.  For a configuration $A$, the {\it initial complex} with respect to $<$ is defined by
$$\Delta ({\rm in}_{<}(I_A)):=\left\{{\rm conv}(B) \ \left| \ B\subset \{\bm{a}_{1},\ldots ,\bm{a}_{n}\}, \prod_{\bm{a}_{i}\in B}x_i\notin \sqrt{{\rm in}_{<}(I_A)}\right.\right\}.$$
It is known that $\Delta ({\rm in}_{<}(I_A))$ is a triangulation of conv$(A)$.
Moreover,  $\Delta ({\rm in}_{<}(I_A))$ is unimodular if and only if  ${\rm in}_{<}(I_A)$ is generated by squarefree monomials.
\begin{Example}
Let $A$ be a configuration
\begin{equation*}
A=(\bm{a}_1 , \bm{a}_2 , \bm{a}_3 , \bm{a}_4 , \bm{a}_5)=
\begin{pmatrix}
0 && 1 && 1 && 0 && 1\\
0 && 1 && 0 && 1 && 1\\
0 && 0 && 1 && 1 && 1\\
1 && 1 && 1 && 1 && 1
\end{pmatrix}.
\end{equation*}
The toric ideal of $A$ is given by $I_A=\langle x_1x_{5}^2-x_2x_3x_4 \rangle$. Let $<$ be the lexicographic order 
on $K[x_1,x_2,x_3,x_4,x_5]$ induced by the ordering $x_1>x_2>x_3>x_4>x_5$. Then, ${\rm in}_<(x_1x_{5}^2-x_2x_3x_4)=x_1x_{5}^2 $, ${\rm in}_<(I_A)=\langle x_1x_{5}^2 \rangle$ and $\sqrt{{\rm in}_<(I_A)}=\langle x_1x_5 \rangle$. 
Thus, the maximal simplices of $\Delta ({\rm in}_{<}(I_A))$ are
$$
\sigma_1 = {\rm conv}(\bm{a}_1, \bm{a}_2, \bm{a}_3, \bm{a}_4)
\mbox{ and }
\sigma_2 = {\rm conv}(\bm{a}_2, \bm{a}_3, \bm{a}_4, \bm{a}_5).
$$
The triangulation $\Delta ({\rm in}_{<}(I_A))$ is not unimodular,
since the normalized volume of $\sigma_1$ is 2. 
\end{Example}
A monomial is said to be a {\it standard monomial} of a toric ideal $I_A$ with respect to a monomial order $<$ if the monomial 
does not belong to in$_{<}(I_A)$.
If  $\Delta ({\rm in}_{<}(I_A))$ is unimodular, then
the number of squarefree standard monomials of degree $i$ corresponds to the number of $(i-1)$-dimensional faces of  $\Delta ({\rm in}_{<}(I_A))$. 
See \cite{HHO, Sturm} for details.

Let $G$ be a graph with $m$ vertices.
We consider the configuration
\begin{equation*}
\mathcal{A}_{{\rm Cut}(G)}=
\begin{pmatrix}
\delta_{A_{1}|B_1} && \delta_{A_{2}|B_2} && \cdots && \delta_{A_{N}|B_N}\\
\\
1 && 1 && \cdots && 1
\end{pmatrix},
\end{equation*}
where $A_i\cap B_i =\emptyset, A_i\cup B_i =V(G)$ for $1\leq{i}\leq{N}$ and $N=2^{m-1}$.
The toric ideal of $\mathcal{A}_{{\rm Cut}(G)}$ is called the {\em cut ideal} of $G$ and is denoted by $I_G$.
The notion of cut ideals was introduced in \cite{SturmfelsSullivant}.
The toric ring and ideal of $\mathcal{A}_{{\rm Cut}(G)}$ were investigated in, e.g., \cite{Eng, Petro, KNP, Ohsugi1, Saka, Shibata}.
The cut ideal $I_{G}$ of a graph $G$ is generated by quadratic binomials if and only if $G$ is $K_4$-minor free \cite{Eng}.
The cut ideal $I_{G}$ of a graph $G$
has a quadratic Gr\"{o}bner basis if $G$ satisfies one of the following:
\begin{itemize}
\item
$G$ is $(K_4, C_5)$-minor free,
where $K_4$ is a complete graph with $4$ vertices,
and $C_5$ is a cycle of length $5$ \cite[Corollary~2.4]{Shibata};
\item
$G$ is a cycle of length $\le 7$ \cite[Theorem~2.3]{Saka}.
\end{itemize}
An {\em unordered partition} $A|B$ of the vertex set $V(G)$
consists of subsets $A,B \subset V(G)$ such that
 $A \cap B =\emptyset , A \cup B = V(G)$.
Given an unordered partition $A|B$, we associate a variable
$q_{A|B}$. 
In particular, $q_{A|B} = q_{B|A}$.
Let $K[q]$ be the polynomial ring in $N=2^{m-1}$ variables over a field $K$
defined by
$$K[q]=K[q_{A_1|B_1},\ldots, q_{A_N|B_N}],$$
where $\{A_1|B_1, \ldots, A_N|B_N\}$ is the set of all unordered partitions of $V(G)$.
Let $<$ be a
reverse lexicographic order \cite[Example $1.8 (b)$]{HHO} on $K[q]$ that satisfies $q_{A|B} < q_{C|D}$ with $\min \{|A|, |B|\} < \min \{|C|, |D|\}$.
A quadratic Gr\"{o}bner basis of the cut ideal of $K_{2,n-2}$ with respect to $<$ is given by the following proposition.

\begin{Proposition}[{\cite[Theorem~2.3]{Shibata}}]
\label{Shibata}
Let $K_{2,n-2}$ be the complete bipartite graph on the vertex set 
$\{1,2\} \cup \{3,\ldots ,n\}$ for $n \ge 4$. 
Then, a Gr\"{o}bner basis of $I_{K_{2,n-2}}$ with respect to $<$ consists of
\begin{enumerate}
\item
$q_{\{1\} \cup A| \{2\} \cup B} q_{\{1\} \cup B| \{2\} \cup A} -q_{\emptyset|[n]}q_{\{1,2\}|\{3,\ldots ,n\}}$,
\item
$q_{A|B}q_{C|D}-q_{{A\cap C}|{B\cup D}}q_{{A\cup C}|{B\cap D}} \ (1\in A\cap C, 2\in B\cap D, A\not\subset C, C\not\subset A)$,
\item
$q_{A|B}q_{C|D}-q_{{A\cap C}|{B\cup D}}q_{{A\cup C}|{B\cap D}} \ (1,2\in A\cap C, A\not\subset C, C\not\subset A)$.
\end{enumerate}
The initial monomial of each binomial 
is the first monomial.
\end{Proposition}

\begin{Example}
Let $K_{2,3}$ be the complete bipartite graph on the vertex set
$\{1,2\} \cup \{3,4,5\}$.
Then, the configuration $\mathcal{A}_{{\rm Cut}(K_{2,3})}$ is 
\begin{equation*}
\mathcal{A}_{{\rm Cut}(K_{2,3})}=
\begin{pmatrix}
0 & 0 & 0 & 0 & 1 & 1 & 1 & 1 & 0 & 0 & 0 & 0 & 1 & 1 & 1 & 1\\
0 & 0 & 0 & 0 & 0 & 0 & 0 & 0 & 1 & 1 & 1 & 1 & 1 & 1 & 1 & 1\\
0 & 0 & 1 & 1 & 1 & 1 & 0 & 0 & 1 & 1 & 0 & 0 & 0 & 0 & 1 & 1\\
0 & 0 & 1 & 1 & 0 & 0 & 1 & 1 & 0 & 0 & 1 & 1 & 0 & 0 & 1 & 1\\
0 & 1 & 0 & 1 & 1 & 0 & 1 & 0 & 1 & 0 & 1 & 0 & 0 & 1 & 0 & 1\\
0 & 1 & 0 & 1 & 0 & 1 & 0 & 1 & 0 & 1 & 0 & 1 & 0 & 1 & 0 & 1\\
1 & 1 & 1 & 1 & 1 & 1 & 1 & 1 & 1 & 1 & 1 & 1 & 1 & 1 & 1 & 1
\end{pmatrix},
\end{equation*}
and a Gr\"{o}bner basis of $I_{K_{2,3}}$ with respect to the monomial order $<$ consists of the following binomials:
\begin{center}
$q_{\{1\}|\{2,3,4,5\}}q_{\{2\}|\{1,3,4,5\}}-q_{\emptyset|\{1,2,3,4,5\}}q_{\{1,2\}|\{3,4,5\}}$,\\
$q_{\{1,3\}|\{2,4,5\}}q_{\{2,3\}|\{1,4,5\}}-q_{\emptyset|\{1,2,3,4,5\}}q_{\{1,2\}|\{3,4,5\}}$,\\
$q_{\{1,4\}|\{2,3,5\}}q_{\{2,4\}|\{1,3,5\}}-q_{\emptyset|\{1,2,3,4,5\}}q_{\{1,2\}|\{3,4,5\}}$,\\
$q_{\{1,5\}|\{2,3,4\}}q_{\{2,5\}|\{1,3,4\}}-q_{\emptyset|\{1,2,3,4,5\}}q_{\{1,2\}|\{3,4,5\}}$,\\
$q_{\{3\}|\{1,2,4,5\}}q_{\{4,5\}|\{1,2,3\}}-q_{\emptyset|\{1,2,3,4,5\}}q_{\{1,2\}|\{3,4,5\}}$,\\
$q_{\{5\}|\{1,2,3,4\}}q_{\{3,4\}|\{1,2,5\}}-q_{\emptyset|\{1,2,3,4,5\}}q_{\{1,2\}|\{3,4,5\}}$,\\
$q_{\{4\}|\{1,2,3,5\}}q_{\{3,5\}|\{1,2,4\}}-q_{\emptyset|\{1,2,3,4,5\}}q_{\{1,2\}|\{3,4,5\}}$,\\
$q_{\{4\}|\{1,2,3,5\}}q_{\{5\}|\{1,2,3,4\}}-q_{\emptyset|\{1,2,3,4,5\}}q_{\{4,5\}|\{1,2,3\}},$\\
$q_{\{3\}|\{1,2,4,5\}}q_{\{5\}|\{1,2,3,4\}}-q_{\emptyset|\{1,2,3,4,5\}}q_{\{3,5\}|\{1,2,4\}},$\\
$q_{\{3\}|\{1,2,4,5\}}q_{\{4\}|\{1,2,3,5\}}-q_{\emptyset|\{1,2,3,4,5\}}q_{\{3,4\}|\{1,2,5\}},$\\
$q_{\{3,5\}|\{1,2,4\}}q_{\{4,5\}|\{1,2,3\}}-q_{\{5\}|\{1,2,3,4\}}q_{\{1,2\}|\{3,4,5\}},$\\
$q_{\{3,4\}|\{1,2,5\}}q_{\{3,5\}|\{1,2,4\}}-q_{\{3\}|\{1,2,4,5\}}q_{\{1,2\}|\{3,4,5\}},$\\
$q_{\{3,4\}|\{1,2,5\}}q_{\{4,5\}|\{1,2,3\}}-q_{\{4\}|\{1,2,3,5\}}q_{\{1,2\}|\{3,4,5\}}$\\
$q_{\{1,4\}|\{2,3,5\}}q_{\{1,5\}|\{2,3,4\}}-q_{\{1\}|\{2,3,4,5\}}q_{\{2,3\}|\{1,4,5\}},$\\
$q_{\{1,3\}|\{2,4,5\}}q_{\{1,5\}|\{2,3,4\}}-q_{\{1\}|\{2,3,4,5\}}q_{\{2,4\}|\{1,3,5\}},$\\
$q_{\{2,3\}|\{1,4,5\}}q_{\{2,4\}|\{1,3,5\}}-q_{\{2\}|\{1,3,4,5\}}q_{\{1,5\}|\{2,3,4\}},$\\
$q_{\{1,3\}|\{2,4,5\}}q_{\{1,4\}|\{2,3,5\}}-q_{\{1\}|\{2,3,4,5\}}q_{\{2,5\}|\{1,3,4\}},$\\
$q_{\{2,3\}|\{1,4,5\}}q_{\{2,5\}|\{1,3,4\}}-q_{\{2\}|\{1,3,4,5\}}q_{\{1,4\}|\{2,3,5\}},$\\
$q_{\{2,4\}|\{1,3,5\}}q_{\{2,5\}|\{1,3,4\}}-q_{\{2\}|\{1,3,4,5\}}q_{\{1,3\}|\{2,4,5\}}$.
\\
\end{center}\vspace{5truept}
The $h^{*}$-polynomial $h^*({\rm Cut}(K_{2,3}),x)$ is
\begin{equation*}
h^*({\rm Cut}(K_{2,3}),x)=x^5+9x^4+26x^3+26x^2+9x+1=(x+1)(x^2+4x+1)^2.
\end{equation*}
\end{Example}

Since the dimension of the cut polytopes of $K_{2,n-2}$ is $2n-4$, the maximum degree of squarefree standard monomials is $2n-3$.
Note that the initial monomial $q_{A|B} q_{C|D}$ of
the binomials in the Gr\"{o}bner basis 
in Proposition~\ref{Shibata} satisfies one of the following conditions:
\begin{itemize}
\item
$1 \in A \cap C$ and $2 \in B \cap D$,
\item
$1 ,2 \in A \cap C$.
\end{itemize}
Hence,
a squarefree monomial
$$q_{A_1|\{1,2\}\cup B_1}\ldots q_{A_k|\{1,2\}\cup B_k} \ 
q_{\{1\}\cup A_{1}'|\{2\}\cup B_{1}'}\ldots q_{\{1\}\cup A_{l}'|\{2\}\cup B_{l}'} \in K[q]$$
is standard if and only if
both 
$q_{A_1|\{1,2\}\cup B_1}\ldots q_{A_k|\{1,2\}\cup B_k}$
and
$q_{\{1\}\cup A_{1}'|\{2\}\cup B_{1}'}\ldots q_{\{1\}\cup A_{l}'|\{2\}\cup B_{l}'}$ are standard.

\begin{Proposition}
\label{standardmonomials}
Each of the squarefree monomials 
\begin{enumerate}
\item[{\rm (1)}]
$q_{A_1|\{1,2\}\cup B_1}\ldots q_{A_k|\{1,2\}\cup B_k}$
\item[{\rm (2)}]
$q_{\{1\}\cup A_{1}'|\{2\}\cup B_{1}'}\ldots q_{\{1\}\cup A_{k}'|\{2\}\cup B_{k}'}$,
\end{enumerate}
of degree $k \leq 2n-3$
is not divisible by the initial monomials if and only if, by changing indices if necessary,
\begin{enumerate}
\item[{\rm (1)}]
$A_1\subsetneq A_2\subsetneq \dots \subsetneq A_k$,
\item[{\rm (2)}]
$A'_1\subsetneq A'_2\subsetneq \dots \subsetneq A'_k \ {\rm and} \ (A'_1, A'_k)\neq (\emptyset ,\{3,\ldots ,n\})$,
\end{enumerate}
respectively.
Moreover, the numbers of squarefree standard monomials of types $(1)$ and $(2)$ above are
\begin{eqnarray}
(k-1)!\genfrac{\{}{\}}{0pt}{}{n-2}{k-1}
+2k!\genfrac{\{}{\}}{0pt}{}{n-2}{k}
+(k+1)!\genfrac{\{}{\}}{0pt}{}{n-2}{k+1},\\
2k!\genfrac{\{}{\}}{0pt}{}{n-2}{k}
+(k+1)!\genfrac{\{}{\}}{0pt}{}{n-2}{k+1},
\end{eqnarray}
respectively, where $\genfrac{\{}{\}}{0pt}{}{n}{k}$ is the Stirling number of the second kind.
\end{Proposition}

\begin{proof}
Let $m_1= q_{A_1|\{1,2\}\cup B_1} \, \ldots \,  q_{A_k|\{1,2\}\cup B_k}$
and $m_2 = q_{\{1\}\cup A_{1}'|\{2\}\cup B_{1}'} \, \ldots \, q_{\{1\}\cup A_{k}'|\{2\}\cup B_{k}'}$ be squarefree monomials.

First, suppose that $m_1$ and $m_2 $ are standard.
Since $m_1$ is squarefree and is not divisible by the initial monomials
$q_{A|B}q_{C|D}$ ($1,2\in A\cap C, A\not\subset C, C\not\subset A$),
we can obtain $B_1 \supsetneq \cdots \supsetneq B_{k-1} \supsetneq B_k$
by changing indices if necessary.
Then, $A_1 \subsetneq \cdots \subsetneq A_{k-1} \subsetneq A_k$.
Similarly, since $m_2$ is squarefree and is not divisible by the initial monomials
$q_{A|B}q_{C|D}$ ($1\in A\cap C, 2\in B\cap D, A\not\subset C, C\not\subset A$),
we can obtain $A_1' \subsetneq \cdots \subsetneq A_{k-1}' \subsetneq A_k'$
by changing indices if necessary.
Moreover, since $m_2$ is not divisible by the initial monomial
$q_{\{1\} | \{2,3,\ldots, n\}} q_{\{1,3,\ldots, n\}| \{2\} }$
of the binomial $q_{\{1\} | \{2,3,\ldots, n\}} q_{\{1,3,\ldots, n\}| \{2\} }-q_{\emptyset|[n]}q_{\{1,2\}|\{3,\ldots ,n\}}
$,
we have $(A'_1, A'_k)\neq (\emptyset ,\{3,\ldots ,n\})$.

Contrarily, suppose that
$m_1$
and $m_2$ satisfy
$A_1\subsetneq A_2\subsetneq \dots \subsetneq A_k$,
$A'_1\subsetneq A'_2\subsetneq \dots \subsetneq A'_k \ {\rm and} \ (A'_1, A'_k)\neq (\emptyset ,\{3,\ldots ,n\})$.
Then, $m_1$ and $m_2$ are not divisible by the initial monomials 
$q_{A|B}q_{C|D}$ ($1,2\in A\cap C, A\not\subset C, C\not\subset A$)
and 
$q_{A|B}q_{C|D}$ ($1\in A\cap C, 2\in B\cap D, A\not\subset C, C\not\subset A$).
Hence, $m_1$ is standard.
Suppose that $m_2$ is not standard.
Then, $m_2$ is divisible by the initial monomial 
$q_{\{1\} \cup A| \{2\} \cup B} q_{\{1\} \cup B| \{2\} \cup A}$.
Since $A'_1\subsetneq A'_2\subsetneq \dots \subsetneq A'_k$.
Thus, we may assume that $A \subsetneq B$.
Since $\{1\} \cup A| \{2\} \cup B$ is a partition,
$A \cup B = \{3,\ldots,n\}$.
Therefore, $(A,B)=(\emptyset, \{3,\ldots,n\})$.
This contradicts the hypothesis $(A'_1, A'_k)\neq (\emptyset ,\{3,\ldots ,n\})$.
Hence, $m_2$ is standard.

The Stirling number of the second kind
$\genfrac{\{}{\}}{0pt}{}{a}{b}$
represents the number of ways to partition a set of $a$ objects into $b$ nonempty subsets. We obtain Table~\ref{four_pattern}
by considering four cases
for the number of squarefree standard monomials
of type (1).
\begin{table}[h] 
\begin{center}
\caption{Number of squarefree standard monomials}  
\begin{tabular}{|c|c|c|}
\hline
{} & {$A_1=\emptyset$} & {$A_1\neq \emptyset$}\\ \hline
{$A_k=\{3,\ldots ,n\}$} & 
{$\displaystyle (k-1)! \genfrac{\{}{\}}{0pt}{}{n-2}{k-1}$} & {$\displaystyle k! \genfrac{\{}{\}}{0pt}{}{n-2}{k}$}\\ \hline
{$A_k\neq \{3,\ldots ,n\}$} & {$\displaystyle k! \genfrac{\{}{\}}{0pt}{}{n-2}{k}$} & {$\displaystyle (k+1)!\genfrac{\{}{\}}{0pt}{}{n-2}{k+1}$}\\ \hline 
\end{tabular}
\label{four_pattern}
\end{center}
\end{table}
\vspace{1truept}
There is a restriction that $(A'_1, A'_k)\neq (\emptyset, \{3,\ldots ,n\})$ for type $(2)$.
Therefore, we obtain the desired number of squarefree standard monomials in each condition.
\end{proof}

\section{$h^{*}$-polynomial of the cut polytope of $K_{2,n-2}$}

Let $\Delta$ be a triangulation of a lattice polytope
$\mathcal{P}$. The {\em $f$-polynomial} $f_{\Delta}(x)=\sum_{i=0}^{d+1}f_{i-1}x^i$ of $\Delta$ encodes the number $f_i$ of $i$-faces for $i=0,1\dots, d$ and $f_{-1}=1$.
The {\em $h$-polynomial} $h_{\Delta}(x)=\sum_{i=0}^{d+1}h_{i}x^i$ is given by the following relation \cite[P.185]{Beck}:
$$h_\Delta(x)=\sum_{i=0}^{d+1}f_{i-1}x^i(1-x)^{d+1-i}.$$
The following is known for 
$h^{*}$-polynomials and $h$-polynomials \cite[Theorem $10.3$]{Beck}.

\begin{Proposition}
If $\mathcal{P}$ is a d-dimensional lattice polytope that admits a unimodular triangulation $\Delta$, then
$\displaystyle h^{*}(\mathcal{P},x)=h_{\Delta}(x).$
\end{Proposition}

We have the following theorem for the $h^{*}$-polynomial of the cut polytope of $K_{2,n-2}$.

\begin{Theorem}
Let  ${\mathcal P} = {\rm Cut}(K_{2,n-2})$ be the cut polytope of $K_{2,n-2}$,
and 
let $\Delta$ be the unimodular triangulation $\Delta({\rm in}_{<}(I_{\mathcal{A}_{\mathcal{P}}}))$ of  ${\mathcal P}$ with respect to the monomial order $<$ in Proposition~\ref{Shibata}. 
Then, the $h^{*}$-polynomial of  ${\mathcal P}$ and the $h$-polynomial of  $\Delta$ are
$$h^{*}({\mathcal P},x) = h_\Delta(x)=(x+1) (A_{n-2}(x))^2,$$
where $A_{n}(x)$ is the Eulerian polynomial of degree $n-1$. In particular, the normalized volume of $\mathcal{P}$ is $h^{*}(\mathcal{P}, 1)=2((n-2)!)^2$.
\end{Theorem}

\begin{proof} 
The Eulerian polynomial $A_{n}(x)$ satisfies
the following condition \cite[Theorem~9.1]{Petersen}:
$$A_{n}(x)=\sum_{k=1}^{n}k! \genfrac{\{}{\}}{0pt}{}{n}{k}
(x-1)^{n-k}.$$
In addition, $\genfrac{\{}{\}}{0pt}{}{n}{k}$ is given by the following equation \cite{Stanley3}:
$$\genfrac{\{}{\}}{0pt}{}{n}{k}=\sum_{m=1}^{k}(-1)^{k-m}m^{n}\binom{k}{m}.$$
The $f$-polynomial of $\Delta$ is given by the following equation:
$$f_\Delta(x)=\sum_{k=0}^{2n-3}f_{k-1}x^{k}.$$
Then, $f_{k-1}$ is equal to the number of squarefree standard monomials of degree $k$.
Recall that
a squarefree monomial
$$q_{A_1|\{1,2\}\cup B_1}\ldots q_{A_\alpha|\{1,2\}\cup B_\alpha} \ 
q_{\{1\}\cup A_{\alpha+1}'|\{2\}\cup B_{\alpha+1}'}\ldots q_{\{1\}\cup A_{k}'|\{2\}\cup B_{k}'} \in K[q]$$
is standard if and only if
$q_{A_1|\{1,2\}\cup B_1}\ldots q_{A_\alpha|\{1,2\}\cup B_\alpha}$
and
$q_{\{1\}\cup A_{\alpha+1}'|\{2\}\cup B_{\alpha+1}'}\ldots q_{\{1\}\cup A_{k}'|\{2\}\cup B_{k}'}$ are standard.
From Proposition~\ref{standardmonomials},
we have $f_{k-1} = \sum_{\alpha=0}^k B_\alpha C_{k-\alpha}$,
where
\begin{eqnarray*}
B_k &=& (k-1)!\genfrac{\{}{\}}{0pt}{}{n-2}{k-1}
+2k!\genfrac{\{}{\}}{0pt}{}{n-2}{k}
+(k+1)!\genfrac{\{}{\}}{0pt}{}{n-2}{k+1},\\
C_k &=& 2k!\genfrac{\{}{\}}{0pt}{}{n-2}{k}
+(k+1)!\genfrac{\{}{\}}{0pt}{}{n-2}{k+1}.
\end{eqnarray*}
Since $B_k = 0$ for any $k \ge n$, and $C_k = 0$ for any $k \ge n-1$,
$$f_\Delta(x)=\sum_{k=0}^{2n-3} \sum_{\alpha=0}^k B_\alpha C_{k-\alpha} x^{k}
=\left(\sum_{k=0}^{n-1}B_{k}x^{k}\right)
\left(\sum_{k=0}^{n-2}C_{k}x^{k}\right).$$
Let $X = x-1$. From \cite[Remark 2.4, Theorem 3.4]{OHSUGI}, $h_{\Delta}(x)$ is a palindromic polynomial of degree $2n-5$. Hence, 
\begin{align}
h_{\Delta}(x)&=x^{2n-5}h_{\Delta}(x^{-1})=x^{2n-5}{\displaystyle \sum_{k=0}^{2n-3}f_{k-1}x^{-k}(1-x^{-1})^{2n-3-k}}=x^{-2}{\displaystyle \sum_{k=0}^{2n-3}f_{k-1}X^{2n-3-k}} \notag \\
&=x^{-2}\left(\sum_{k=0}^{n-1}B_{k}X^{n-k-1}\right)\left(\sum_{k=0}^{n-2}C_{k}X^{n-k-2}\right).\notag
\end{align}
It then follows that
\begin{align}
\sum_{k=0}^{n-1}B_{k}X^{n-k-1}
&=\sum_{k=1}^{n}B_{k-1}X^{(n-1)-(k-1)} \notag \\
&=\sum_{k=3}^{n}(k-2)!\genfrac{\{}{\}}{0pt}{}{n-2}{k-2}X^{(n-2)-(k-2)}\notag \\
& \ \ \ \ \ \ +2X\sum_{k=2}^{n-1}(k-1)!\genfrac{\{}{\}}{0pt}{}{n-2}{k-1}X^{(n-2)-(k-1)}+X^2\sum_{k=1}^{n-2}k!\genfrac{\{}{\}}{0pt}{}{n-2}{k}X^{(n-2)-k} \notag\\
&=\sum_{k'=1}^{n-2}k'!
\genfrac{\{}{\}}{0pt}{}{n-2}{k'}X^{(n-2)-k'}\notag \\
& \ \ \ \ \ \ +2X\sum_{k''=1}^{n-2}k''!\genfrac{\{}{\}}{0pt}{}{n-2}{k''}X^{(n-2)-k''} +X^2\sum_{k=1}^{n-2}k!\genfrac{\{}{\}}{0pt}{}{n-2}{k}X^{(n-2)-k}\notag \\
&=A_{n-2}(x)+2XA_{n-2}(x)+X^2A_{n-2}(x)\notag \\
&=x^2A_{n-2}(x),\notag\\
\sum_{k=0}^{n-2}C_{k}X^{n-k-2} &=\sum_{k=0}^{n-2}
\left( 2k!\genfrac{\{}{\}}{0pt}{}{n-2}{k}X^{n-k-2}+(k+1)!\genfrac{\{}{\}}{0pt}{}{n-2}{k+1}X^{n-k-2} \right) \notag \\
&=2\sum_{k=1}^{n-2}k!\genfrac{\{}{\}}{0pt}{}{n-2}{k}X^{(n-2)-k}+X\sum_{k=0}^{n-3}(k+1)!\genfrac{\{}{\}}{0pt}{}{n-2}{k+1}X^{(n-3)-k}\notag \\
&=2A_{n-2}(x)+X\sum_{k'=1}^{n-2}k'!\genfrac{\{}{\}}{0pt}{}{n-2}{k'}X^{(n-2)-k'}\notag \\
&=2A_{n-2}(x)+X A_{n-2}(x)\notag \\
&=(1+x)A_{n-2}(x).\notag 
\end{align}
Therefore, $h^{*}({\mathcal P},x)=h_\Delta (x)=(x+1)(A_{n-2}(x))^2$.
\end{proof}

\section{Acknowledgment}
The author is grateful to the anonymous referees for their careful reading and useful comments.

\bigskip
\hrule
\bigskip

\noindent 2010 \textit{Mathematics Subject Classification}:
Primary 05A15; Secondary 13P10, 52B20.

\noindent \textit{Keywords}:
cut polytope, complete bipartite graph, $h^*$-polynomial,  Gr\"{o}bner basis

\bigskip
\hrule
\bigskip

\end{document}